\documentclass[12pt,a4paper]{article}
\usepackage{latexsym,amssymb,amsfonts,amsmath,amsthm,nccmath,float,enumitem,setspace,bm,authblk,appendix,makecell}
\usepackage[usenames,dvipsnames]{xcolor}
\usepackage[hidelinks]{hyperref}
\usepackage[margin=1.8cm]{geometry}
\usepackage{cleveref}

\hypersetup{
	colorlinks,
	linkcolor={red!80!black},
	citecolor={blue!80!black},
	urlcolor={blue!80!black}
}

\newtheorem{Theorem}{Theorem}[section]

\newtheorem{Lemma}[Theorem]{Lemma}

\newtheorem*{Problem}{Problem}
\newtheorem*{Induction hypothesis}{Induction hypothesis}

\newcommand{\A}{\mathcal{A}}
\newcommand{\B}{\mathcal{B}}
\newcommand{\F}{\mathcal{F}}

\newcommand{\M}{\mathcal{M}}
\renewcommand{\P}{\mathcal{P}}

\renewcommand{\O}{\Omega}

\numberwithin{equation}{section}

\allowdisplaybreaks

\let\oldproofname=\proofname
\renewcommand{\proofname}{\rm\bf{\oldproofname}}

\allowdisplaybreaks[4]

\begin{document}
	
	\title{Erd\H{o}s--Ko--Rado-type problem for hypergraph matchings}	
	\author[a]{Binwei Zhao}
	\author[a,b]{Tao Feng}
	\author[c]{Xiaomiao Wang}
	\author[d]{Menglong Zhang}
	
	\affil[a]{School of Mathematics and Statistics, Beijing Jiaotong University, Beijing 100044, P.R. China}
	\affil[b]{Hebei Provincial Key Laboratory of Mathematical Theory and Analysis for Network and Data Science, Beijing Jiaotong University, Beijing, 100044, P.R. China}
	\affil[c]{School of Mathematics and Statistics, Ningbo University, Ningbo 315211, P.R. China}
	\affil[d]{Institute of Mathematics and Interdisciplinary Sciences, Xidian University, Xi'an 710126, P.R. China}
	
	\renewcommand*{\Affilfont}{\small\it}
	\renewcommand\Authands{ and }
	
	\affil[ ]{binweizhao@qq.com; tfeng@bjtu.edu.cn; wangxiaomiao@nbu.edu.cn; mlzhang@bjtu.edu.cn}
	
	\date{}
	\maketitle
	
\begin{abstract}
Given integers $1\leq t\leq k$, a family of $k$-matchings in a complete $r$-partite $r$-uniform hypergraph is said to be $t$-intersecting if any two of its members share at least $t$ common edges. This concept unifies several well-studied classes of intersecting families, including classical intersecting families, intersecting families of permutations, partial permutations, and generalized permutations, as well as intersecting families of injections. In this paper we employ two approaches to determine the maximum size of $t$-intersecting families of $k$-matchings and to characterize the extremal families that attain this bound. Using a recent result of Keller, Lifshitz, Minzer, and Sheinfeld on $t$-intersecting families of permutations, we obtain Erd\H{o}s--Ko--Rado-type theorems whose thresholds depend only on $t$. We also develop a $t$-cover-based approach that offers a complementary characterization of the extremal families.
\end{abstract}
	
	
\footnotetext{Supported by NSFC under Grant 12271023, and Ningbo Natural Science Foundation under Grant 2024J018}

\section{Introduction}
	
Let $\Omega$ be a nonempty set, and let $t$ and $k$ be integers satisfying $1 \leq t \leq k \leq |\Omega|$. Denote by $2^{\Omega}$ the power set of $\Omega$ and by $\binom{\Omega}{k}$ the family of all $k$-subsets of $\Omega$. A family $\F\subseteq 2^{\Omega}$ is said to be \emph{$t$-intersecting} if $|F_1\cap F_2|\geq t$ for all $F_1,F_2\in\F$; when $t=1$, $\F$ is simply called \emph{intersecting}. A central problem in extremal set theory is to determine the maximum size of a $t$-intersecting subfamily of a given family $\mathcal G \subseteq 2^{\Omega}$, and to characterize the subfamilies that achieve this size. A family $\mathcal{S}\subseteq \mathcal G \subseteq 2^{\O}$ is called a \emph{$t$-star} of $\mathcal G$ if it consists of all elements of $\mathcal G$ containing $t$ specific elements of $\O$; the case $t=1$ is simply termed a \emph{star}.
	
For example, let $\Omega=[n]:=\{1,2,\dots,n\}$. The celebrated Erd\H{o}s--Ko--Rado (EKR) Theorem \cite{EKR1961} shows that given $1\leq t\leq k$, for any integer $n>n_0(k,t)=(k-t)\binom{k}{t}^3+t$, if $\F\subseteq\binom{[n]}{k}\subseteq 2^{[n]}$ is a $t$-intersecting family, then $|\F|\leq \binom{n-t}{k-t}$, and equality holds if and only if $\F$ is a $t$-star of $\binom{[n]}{k}$. The smallest possible $n_0(k,t)$ has been proved to be $(t+1)(k-t+1)$ by Erd\H{o}s, Ko and Rado \cite{EKR1961} for $t=1$, and by Frankl for $t\geq 15$ \cite{frankl1978erdos}, and later by Wilson for all $t\geq 1$ \cite{wilson1984exact}.
\begin{Theorem}\emph{\cite{EKR1961,frankl1978erdos,wilson1984exact}}\label{thm:classical EKR}
	Let $\F\subseteq \binom{[n]}{k}$ be a $t$-intersecting family with $1\leq t\leq k\leq n$. If $n\geq (t+1)(k-t+1)$, then $|\F|\leq \binom{n-t}{k-t}$. Furthermore, if $n>(t+1)(k-t+1)$, then $|\F|=\binom{n-t}{k-t}$ if and only if $\F$ is a $t$-star of $\binom{[n]}{k}$.
\end{Theorem}

Up to now, there have been a large number of results on extensions of the EKR Theorem by taking $\Omega$ to be different combinatorial objects, not just for $\Omega=[n]$. We call these {\em EKR-type results}, and call the corresponding problems {\em EKR-type problems}. The reader is referred to \cite{FT16} for more details. In this paper, we consider EKR-type problem for hypergraph matchings.
	
Let $r\geq 1$ and $n_1,\dots,n_r$ be positive integers. Let $t$ and $k$ be integers such that $1\leq t\leq k\leq \min_{1\leq i\leq r} n_i$. Define the \emph{complete $r$-partite $r$-uniform hypergraph} $\mathcal{K}=\mathcal{K}(n_1,\dots,n_r)=(V,E)$ as the hypergraph whose vertex set $V$ is partitioned into $r$ parts $V_1,\ldots,V_r$ with $|V_i|=n_i$, and whose edge set consists of all $r$-element sets that contain exactly one vertex from each part, i.e., $E=\{\{x_1,\dots,x_r\}:x_i\in V_i,i\in[r]\}$. A \emph{$k$-matching} in $\mathcal{K}$ is a set of $k$ pairwise vertex-disjoint edges from $E$. Let $\mathcal{M}_k(n_1,\dots,n_r)$ be the set of all $k$-matchings of $\mathcal{K}$. Then $|\M_{k}(n_1,\dots,n_r)|=(k!)^{r-1}\prod_{i=1}^{r}\binom{n_i}{k}$. Let $\O=E(\mathcal{K})$. A subfamily $\F\subseteq \mathcal{M}_k(n_1,\dots,n_r)\subseteq 2^{\O}$ is said to be \emph{$t$-intersecting} if $|F_1\cap F_2|\geq t$ for any two $k$-matchings $F_1,F_2\in \F$. The size of a $t$-star of $\F\subseteq \M_{k}(n_1,\dots,n_r)$ is at most
\begin{align}\label{eq:20251205}
    ((k-t)!)^{r-1}\prod_{i=1}^{r}\binom{n_i-t}{k-t}.
\end{align}
	
The following EKR-type problem was introduced in \cite{Mammoliti2018}.
	
\begin{Problem}
	Determine the maximum-sized $t$-intersecting subfamilies of $\mathcal{M}_k(n_1,\dots,n_r)$.
\end{Problem}
	
If $r=1$, then $\mathcal{M}_k(n_1)$ can be regarded as $\binom{[n_1]}{k}$. This case corresponds to Theorem \ref{thm:classical EKR}.
	
If $r=2$ and $k=n_1=n_2=n$, then $\M_n(n,n)$ can be viewed as the set of all permutations of $[n]$, where each $n$-matching is equivalent to a permutation of $[n]$. This case corresponds to the EKR problem for permutations. Specifically, two permutations $\sigma_1$ and $\sigma_2$ on $[n]$ are called \emph{$t$-intersecting} if $|\{i\in [n]:\sigma_1(i)=\sigma_2(i)\}|\geq t$. A family $\F$ of some permutations on $[n]$ is \emph{$t$-intersecting} if any two permutations in $\F$ are $t$-intersecting. For $t$=1, Frankl and Deza \cite{frankl1977maximum} demonstrated that for any positive integer $n$, if $\F\subseteq\M_n(n,n)$ is intersecting then $|\mathcal{F}|\leq (n-1)!$. Cameron and Ku \cite{cameron2003intersecting}, and Larose and Malvenuto \cite{larose2004stable} independently proved that an intersecting family $\mathcal{F}\subseteq\M_n(n,n)$ has the maximum size $(n-1)!$ if and only if $\F$ is a star of $\M_n(n,n)$. For general $t$, Frankl and Deza \cite{frankl1977maximum} conjectured that there exists $n_0(t)$ such that for any $n\geq n_0(t)$, if $\mathcal{F}\subseteq\M_n({n,n})$ is $t$-intersecting, then $|\mathcal{F}|\leq (n-t)!$. This conjecture was confirmed by Ellis, Friedgut, and Pilpel \cite{ellis2011intersecting}, who further conjectured that for any $n>2t$, the $t$-stars of $\M_n(n,n)$ are the unique maximum-sized $t$-intersecting subfamilies. Recently, Keller, Lifshitz, Minzer and Sheinfeld \cite{permutations_n_linear_t} proved this conjecture for any $n\geq C_0t$ where $C_0$ is an absolute constant. They also provided a stability result for the structure of $\F$ with $|\F|\geq 0.75(n-t)!$.

	\begin{Theorem}\emph{\cite[Theorem 1]{permutations_n_linear_t}}\label{thm:KLMS}
		Let $\F\subseteq \M_n(n,n)$ be a $t$-intersecting family with $1\leq t\leq n$. Then there exists an absolute constant $C_0>0$ such that the following hold for all $n\geq C_0t$.
		\begin{itemize}
			\item[\em(1)] $|\F|\leq (n-t)!$.
			\item[\em(2)] If $|\F|\geq 0.75(n-t)!$, then there exists some $T\in\M_t(n,n)$ such that $\F\subseteq \P_n(T)$, where $\P_n(T)=\{P\in \M_{n}(n,n):T\subseteq P\}$.
		\end{itemize}
	\end{Theorem}	
	
If $r=2$ and $k=n_1\leq n_2=n$, then a $k$-matching of $\mathcal{K}(k,n)$ can be reformulated as an \emph{injection} from $[k]$ to $[n]$, and $\M_{k}(k,n)$ is the set of all injections from $[k]$ to $[n]$. Brunk and Huczynska \cite{brunk2010EKRinjection} established the following result.
	
\begin{Theorem} \label{Brunk_injection_theorem}
Let $\F\subseteq \mathcal{M}_k(k,n)$ be a $t$-intersecting family with $1\leq t\leq k\leq n$.
\begin{enumerate}
\item[$(1)$] \emph{\cite[Theorem 2.1 and Corollary 2.10]{brunk2010EKRinjection}} If $t=1$, then $|\F|\leq \frac{(n-1)!}{(n-k)!}$, and equality holds if and only if $\F$ is a star of $\M_{k}(k,n)$.
\item[$(2)$] \emph{\cite[Corollary 3.3]{brunk2010EKRinjection}} If $1< t< k$ and $n>t+\frac{k!}{(k-t-1)!}$, then $|\F|\leq \frac{(n-t)!}{(n-k)!}$, and equality holds if and only if $\F$ is a $t$-star of $\M_{k}(k,n)$.
\end{enumerate}
\end{Theorem}
	
Our first result extends Theorem \ref{thm:KLMS} to the case of $k\leq n$.

\begin{Theorem}\label{thm:main-result-injections}
Let $\F\subseteq \M_{k}(k,n)$ be a $t$-intersecting family with $1\leq t\leq k\leq n$. Let $C_0$ be a constant such that Theorem \ref{thm:KLMS} holds. If $n\geq C_0t$, then $|\F|\leq \frac{(n-t)!}{(n-k)!}$, and equality holds if and only if $\F$ is a $t$-star of $\M_{k}(k,n)$.
\end{Theorem}	
	
If $r=2$ and $k<n_1=n_2=n$, a $k$-matching of $\mathcal{K}(n,n)$ is called a {\em $k$-partial permutation}. Ku and Leader \cite{ku2006erdHos} showed that for any $1\leq k<n$, if $\mathcal{F}\subseteq\M_{k}(n,n)$ is intersecting, then $|\mathcal{F}|\leq \binom{n-1}{k-1}^2(k-1)!$, and for $8 \leq k \leq n-3$, equality holds if and only if $\mathcal{F}$ is a star of $\mathcal{M}_k(n, n)$. The characterization for the extremal case was later extended to all $1\leq k<n$ by Li and Wang \cite{li2007erdHos}. For general $t$, Ku \cite[Theorem 6.6.6]{ku2004intersecting} proved that for any $1\leq t\leq k$	and sufficiently large $n$ (relative to $k$ and $t$), if $\mathcal{F}\subseteq\M_{k}(n,n)$ is $t$-intersecting, then $|\mathcal{F}|\leq \binom{n-t}{k-t}^2(k-t)!$, with equality if and only if $\F$ is a $t$-star of $\M_{k}(n,n)$. This result was subsequently improved by Borg \cite{borg2009t}, who weakened the condition on $n$ to $n\geq \binom{k}{t}\binom{3k-2t-1}{\lfloor\frac{3k-2t-1}{2}\rfloor}\frac{k!}{(k-t-1)!}+k+1$. Later, Borg \cite{borg2017max} further improved the bound to $n\geq \binom{k}{t}k+t+1$.
	
\begin{Theorem}\label{k_partial_permutations}
Let $\F\subseteq \mathcal{M}_k(n,n)$ be a $t$-intersecting family with $1\leq t\leq k< n$.
\begin{enumerate}
\item[$(1)$] \emph{\cite{ku2006erdHos,li2007erdHos}} If $t=1$, then $|\mathcal{F}|\leq \binom{n-1}{k-1}^2(k-1)!$, and equality holds if and only if $\F$ is a star of $\M_{k}(n,n)$.
\item[$(2)$] \emph{\cite[Theorem 4.12]{borg2017max}} If $1< t\leq k$ and $n\geq \binom{k}{t}k+t+1$, then $|\F|\leq \binom{n-t}{k-t}^2(k-t)!$, and equality holds if and only if $\F$ is a $t$-star of $\M_{k}(n,n)$.
\end{enumerate}
\end{Theorem}
	
If $r=2$ and $k<n_1\leq n_2$, then a $k$-matching of $\mathcal{K}(n_1,n_2)$ is said to be a \emph{generalized permutation}. Borg and Meagher \cite{borg2015generalizedpermutation} proved that if $k\leq \min\{n_1,n_2\}$ and $\F\subseteq\M_{k}(n_1,n_2)$ is intersecting, then $|\F|\leq \binom{n_1-1}{k-1}\binom{n_2-1}{k-1}(k-1)!$ and equality holds if and only if $\F$ is a star of $\M_{k}(n_1,n_2)$. For general $r\geq 2$, Mammoliti \cite{Mammoliti2018} established the following result.

\begin{Theorem}\label{Mammoliti_main_result}
Let $\F\subseteq\M_{k}(n_1,n_2,\dots,n_r)$ be a $t$-intersecting family with $1\leq t\leq k\leq \min_{1\leq i\leq r} n_i$ and $r\geq 2$.
\begin{enumerate}
\item[$(1)$] {\em \cite[Theorem 1.3]{Mammoliti2018}} If $t=1$, then $|\F|\leq ((k-1)!)^{r-1}\prod_{i=1}^{r}\binom{n_i-1}{k-1}$, and equality holds if and only if $\F$ is a star of $\M_{k}(n_1,\dots,n_r)$.
\item[$(2)$] {\em \cite[Theorem 3.3]{Mammoliti2018}} Let $n_2$ be the second smallest $n_i$ for $1\leq i\leq r$. If $t>1$ and $n_2$ is sufficiently larger than $k$ and $t$, then $|\F|\leq \left((k-t)!\right)^{r-1}\prod_{i=1}^{r}\binom{n_i-t}{k-t}$, and equality holds if and only of $\F$ is a $t$-star of $\M_{k}(n_1,\dots,n_r)$.
\end{enumerate}
\end{Theorem}
	
Our second result replaces the previous lower bounds of $n_2$ depending on both $k$ and $t$ in Theorems \ref{k_partial_permutations} and \ref{Mammoliti_main_result} with a new unified lower bound of $n_2$ depending on $t$ and $C_0$, where $C_0$ is a constant such that Theorem \ref{thm:KLMS} holds.
	
\begin{Theorem}\label{thm:main-result}
Let $\F\subseteq \M_{k}(n_1,n_2,\dots,n_r)$ be a $t$-intersecting family with $1\leq t\leq k\leq n_1\leq n_2\leq\dots\leq n_r$ and $r\geq 2$. Let $C_0$ be a constant such that Theorem \ref{thm:KLMS} holds and let $R(t) = \max\{C_0(t+1),t+4((t+1)^t-1)\}$. If $n_2\geq R(t)$, then $|\F|\leq \left((k-t)!\right)^{r-1}\prod_{i=1}^{r}\binom{n_i-t}{k-t}$, and equality holds if and only if $\F$ is a $t$-star of $\M_{k}(n_1,n_2,\dots,n_r)$.
\end{Theorem}
	
Our third result, based on a $t$-cover argument, refines Theorem \ref{Brunk_injection_theorem}, Theorem \ref{k_partial_permutations}, and Theorem \ref{Mammoliti_main_result} in the case $t>1$. We state these refinements as three separate theorems below.
	
\begin{Theorem}\label{main_result_1}
Let $\F\subseteq \mathcal{M}_k(k,n)$ be a $t$-intersecting family with $1< t< k$ and $n>(t+1)(k-t+1)+t$. Then $|\F|\leq \frac{(n-t)!}{(n-k)!}$, and equality holds if and only if $\F$ is a $t$-star of $\M_k(k,n)$.
\end{Theorem}
	
\begin{Theorem}\label{main_result_2}
Let $\F\subseteq \mathcal{M}_k(n_1,n_2)$ be a $t$-intersecting family with $1< t< k$ and $\min\{n_1,n_2\}\geq \sqrt{(k-t+1)(k-t)(t+2)}+t$. Then $|\F|\leq \binom{n_1-t}{k-t}\binom{n_2-t}{k-t}(k-t)!$, and equality holds if and only if $\F$ is a $t$-star of $\M_k(n_1,n_2)$.
\end{Theorem}
	
\begin{Theorem}\label{main_result_3}
Let $r\geq \log_2((t+1)(k-t+1)^2)$ and $1< t< k< \min_{1\leq i\leq r} n_i$. Let $\F\subseteq \mathcal{M}_k(n_1,\dots,n_r)$ be a $t$-intersecting family. Then $|\F|\leq \left((k-t)!\right)^{r-1}\prod_{i=1}^{r}\binom{n_i-t}{k-t}$, and equality holds if and only if $\F$ is a $t$-star of $\M_k(n_1,\dots,n_r)$.
\end{Theorem}

Although the thresholds in Theorems \ref{main_result_1}--\ref{main_result_3} depend on both $k$ and $t$, these results offer complementary advantages over Theorems \ref{thm:main-result-injections} and \ref{thm:main-result}. First, the condition $n>(t+1)(k-t+1)+t$ in Theorem \ref{main_result_1} is weaker than $n\ge C_0t$ whenever $k < t-1+\frac{t(C_0-1)}{t+1}$ (note that $C_0$ is large by inequality (5.1) in \cite{permutations_n_linear_t}). A similar comparison applies between Theorem \ref{thm:main-result} (for $r=2$) and Theorem \ref{main_result_2}. Second, for $r\ge 3$, Theorem \ref{thm:main-result} requires $n_2,\dots,n_r\ge R(t)$ to force extremal families to be $t$-stars. In contrast, Theorem \ref{main_result_3} shows that the same conclusion holds whenever $r$ is sufficiently large relative to $k$ and $t$, even if each part has size as small as $k+1$.

The rest of this paper is organized as follows. Section \ref{section_preliminaries} contains the necessary preliminaries and proves Theorem \ref{thm:main-result-injections}. Sections \ref{section_GP} and \ref{section_HP} together provide the proof of Theorem \ref{thm:main-result}. Section \ref{section_t-covers} introduces a $t$-cover approach and proves Theorems \ref{main_result_1}--\ref{main_result_3}. Section \ref{section_concluding} offers some concluding remarks.

\section{Proof of Theorem \ref{thm:main-result-injections}}\label{section_preliminaries}

Given positive integers $a,b$ with $a\leq b$. To prove Theorem \ref{thm:main-result-injections}, we may bridge $\M_a(a,b)$ and $\M_{b}(b,b)$. We identify an edge of $\mathcal{K}(a,b)$ with an ordered pair $(x,y)\in [a]\times[b]$. Hence an element of $\M_a(a,b)$ is a subset $F\subseteq [a]\times[b]$ of size $a$.
	
Let $\mathcal{A}$ and $\mathcal{B}$ be two finite sets and let $f:\A \to \B$ be a mapping. For any $B\in \mathcal{B}$, the {\em fiber} of $f$ over $B$, denoted by $f^{-1}(B)$, is the set of all elements in $\A$ that map to $B$, i.e., $f^{-1}(B)=\{A\in \A:f(A)=B\}$. If $B_1,B_2\in \B$ and $B_1\neq B_2$, then $f^{-1}(B_1)\cap f^{-1}(B_2)=\varnothing$. Let $f(\mathcal{A})=\{f(A):A\in \A\}$. Then
\begin{align}\label{eq:property of fibers}
	\A=\bigsqcup_{B\in f(\A)}f^{-1}(B),~\text{where $\bigsqcup$ represents the disjoint union}.
\end{align}
The following lemma is straightforward and we omit its proof.
	
\begin{Lemma}\label{lem:fibers}
Let $\F\subseteq \M_{a}(a,b)$ be a $t$-intersecting family with $1\leq t\leq a\leq b$. Define a mapping $
\gamma:\M_{b}(b,b)\to\M_{a}(a,b)$ satisfying that $\gamma(P)=P\cap([a]\times[b])$ for any $P\in\M_{b}(b,b)$. Then the following hold.
\begin{enumerate}
\item[\em(1)] For any $F\in\M_{a}(a,b)$, $\gamma^{-1}(F)=\{P\in\M_{b}(b,b):F\subseteq P\}$.
\item[\em(2)] The family $\gamma^{-1}(\F):=\{P\in\M_{b}(b,b):\gamma(P)\in \F\}$ is $t$-intersecting.
\item[\em(3)] $\gamma^{-1}(\F)=\bigsqcup_{F\in\F}\gamma^{-1}(F)$ and $|\gamma^{-1}(\F)|=(b-a)!|\F|$.
\end{enumerate}
\end{Lemma}

We shall make use of the following theorem of Hall, which provides a necessary and sufficient condition for the existence of a system of distinct representatives. For our purpose, we state it in the following set-theoretic form.
	
\begin{Lemma}[Hall's Marriage Theorem \cite{Hall_1935_marriage}]\label{thm:Hall's theorem-set version}
Let $Y$ be a nonempty set and let $A_1,A_2,\dots,A_m\subseteq Y$. Then there exist pairwise distinct elements $y_1,y_2,\dots,y_m\in Y$ such that $y_i\in A_i$ for $1\le i\leq m$ if and only if for every subset $I\subseteq [m]$, $\left| \bigcup_{i\in I}A_i \right| \geq |I|$.
\end{Lemma}
	
\begin{Lemma}\label{lem:auxiliary_1}
Let $D$ and $Y$ be finite nonempty sets with $|D|\leq |Y|$ and $|Y|\geq 2$. Let $X\subseteq D$, and let $\psi:X\to Y$ be an injection.  Then there exists an injection $\xi:D\to Y$ such that $\xi(d) \neq \psi(d)$ for any $d\in X$.
\end{Lemma}

\begin{proof}
Let $D=\{d_1,\dots,d_m\}$. For $1\leq i\leq m$, define
\begin{align*}
A_i=\begin{cases}
		Y\setminus\{\psi(d_i)\}, & \text{if } d_i\in X; \\
		Y, & \text{if }  d_i\notin X.
	\end{cases}
\end{align*}
We shall show that for every subset $I\subseteq [m]$, $\left| \bigcup_{i\in I}A_i \right| \geq |I|$. The case $I=\varnothing$ is trivial, so we assume that $I\neq\varnothing$. If there exists $j\in I$ such that $d_j\notin X$, then $A_j=Y$. So $\bigcup_{i\in I}A_i=Y$ and $\left|\bigcup_{i\in I}A_i\right|=|Y|\geq |D|=m\geq |I|$.
		
Now suppose that $d_i\in X$ for every $i\in I$. If $|I|=1$, then $\left|\bigcup_{i\in I}A_i\right|=|Y|-1\geq 2-1=1=|I|$. If $|I|\geq 2$, take two distinct $i_1,i_2\in I$. Then $A_{i_1}=Y\setminus\{\psi(d_{i_1})\}$ and $A_{i_2}=Y\setminus\{\psi(d_{i_2})\}$. Since $\psi:X\to Y$ is an injection, $\psi(d_{i_1})\neq \psi(d_{i_2})$, and hence $\psi(d_{i_1})\in A_{i_2}$, which leads to $A_{i_1}\cup A_{i_2}=Y$. Therefore, $\left|\bigcup_{i\in I}A_i\right|=|Y|\geq |D|=m\geq |I|$.
		
Therefore, for every subset $I\subseteq [m]$, $\left| \bigcup_{i\in I}A_i \right| \geq |I|$. By Lemma \ref{thm:Hall's theorem-set version}, there exist pairwise distinct $y_1,y_2,\dots,y_m\in Y$ such that $y_i\in A_i$ for all $1\leq i\leq m$. Since $|D|\leq |Y|$, we can define an injection $\xi:D\to Y$ by $\xi(d_i)=y_i$ for $1\leq i\leq m$. For any $d_i\in X\subseteq D$, we have $y_i\in Y\setminus\{\psi(d_i)\}$, and so $y_i=\xi(d_i)\neq \psi(d_i)$. Hence $\xi$ is our desired injection.
	\end{proof}

\begin{Lemma}\label{lem:projection}
Let $r,t,k,n_1,n_2,\dots,n_r$ be integers with $r\geq 2$ and $1\leq t\leq k\leq n_1\leq n_2\leq\cdots\leq n_r$. Let $F\in \M_{k}(n_1,n_2,\dots,n_r)$ and let $T\subseteq F$ with $|T|=t$. If $n_r-t\geq 2$, then there exists $F^*\in \M_{k}(n_1,n_2,\dots,n_r)$ such that $F\cap F^*=T$.
\end{Lemma}

\begin{proof}
For convenience, we represent each edge $\{x_1,\dots,x_r\}$ with $x_i\in V_i$ in $\mathcal{K}(n_1,\dots,n_r)$ by the ordered tuple $(x_1,\dots,x_r)$, where $V_i$ is the $i$-th part of the vertex set with $|V_i|=n_i$. Let $D=F\setminus T$. Define $V_r(T) = \{x_r\in V_r: (x_1,x_2,\dots,x_r)\in T\}$. Since $T$ is a matching, the $t$ edges of $T$ use $t$ distinct vertices in $V_r$, so $|V_r(T)|=t$. Let $Y=V_r\setminus V_r(T)$. Then $k-t=|D|\leq |Y|=n_r-t$ and $|Y|\geq 2$.
		
For each edge ${\bm e}=(x_1,x_2,\dots,x_r)\in D$, define $\psi({\bm e})=x_r$. Since $F$ is a matching, $\psi$ is an injection from $D$ to $Y$. By Lemma \ref{lem:auxiliary_1}, there exists an injection $\xi:D\to Y$ such that $\xi({\bm e})\neq \psi({\bm e})$ for all ${\bm e}\in D$. Now for each edge ${\bm e}=(x_1,x_2,\dots,x_{r-1},x_r)\in D$, define a new edge ${\bm e}^* = (x_1,x_2,\dots,x_{r-1},\xi({\bm e}))$ and let $F^* = T \cup \{{\bm e}^*:{\bm e}\in D\}$. The last coordinates of the new edges are pairwise distinct and lie in $Y$, while their other coordinates are inherited from $D$ and hence are disjoint from $T$ and from each other. Therefore $F^*\in\M_{k}(n_1,n_2,\dots,n_r)$, and clearly $T\subseteq F^*$. It remains to check $F\cap F^*=T$. For any $\mathbf e\in D$, the new edge $\mathbf e^*$ shares the first coordinate with $\mathbf e$. Since $F$ is a matching, the only edge in $F$ that can share a vertex with $\mathbf e$ is $\mathbf e$ itself. But $\xi(\mathbf e)\ne x_r$, so $\mathbf e^*\ne \mathbf e$; hence $\mathbf e^*\notin F$. Thus no new edge belongs to $F$, and the edges of $T$ are retained, so
$F\cap F^*=T$.
\end{proof}
	
Now we are ready to prove Theorem \ref{thm:main-result-injections}.

\begin{proof}[\bf Proof of Theorem \ref{thm:main-result-injections}]
If $k=n$, then the result follows immediately from Theorem \ref{thm:KLMS}. If $t=k$, then any $k$-intersecting subfamily of $\M_{k}(k,n)$ contains at most one element, and hence the conclusion is trivial. Now we assume that $1\leq t<k<n$, which implies $n-t\geq 2$.
		
Since $\F\subseteq\M_{k}(k,n)$ is $t$-intersecting, by Lemma \ref{lem:fibers}(2) with $a=k$ and $b=n$, $\gamma^{-1}(\F)\subseteq\M_{n}(n,n)$ is $t$-intersecting. Since $n\geq C_0t$, by Theorem \ref{thm:KLMS}(1), $|\gamma^{-1}(\F)|\leq (n-t)!$. Combining this with Lemma \ref{lem:fibers}(3), we have
\begin{align*}
    |\F| = \frac{|\gamma^{-1}(\F)|}{(n-k)!} \leq \frac{(n-t)!}{(n-k)!}.
\end{align*}
Now suppose that $|\F|=\frac{(n-t)!}{(n-k)!}$. Then $|\gamma^{-1}(\F)|=(n-t)!$. Since $(n-t)!>0.75(n-t)!$, by Theorem \ref{thm:KLMS}(2), there exists some $T\in\M_{t}(n,n)$ such that $\gamma^{-1}(\F)=\mathcal{P}_n(T)=\{P\in\M_{n}(n,n):T\subseteq P\}$. It remains to show that $\F=\{F\in\M_{k}(k,n):T\subseteq F\}$.

Given any $F\in\F$, we will prove that $T\subseteq F$. Suppose not. Then there exists $F^*\in\F$ such that $T\not\subseteq F^*$. By Lemma \ref{lem:fibers}(1), there exists $P^*\in\gamma^{-1}(\F)$ such that $F^*\subseteq P^*$. Since every element of $\gamma^{-1}(\F)$ contains $T$, we have $T\subseteq P^*$. By Lemma \ref{lem:projection} with $r=2$, there exists $P'\in\M_{n}(n,n)$ such that $P^*\cap P'=T$. Since $T\subseteq P'$, we have $P'\in\gamma^{-1}(\F)$, and so there exists $F'\in\F$ with $\gamma(P')=F'$ and $F'\subset P'$. Hence $F^*\cap F'\subseteq P^*\cap P'=T$. Since $T\not\subseteq F^*$, this gives $|F^*\cap F'|<t$, a contradiction to the $t$-intersecting property of $\F$. Thus for any $F\in\F$, we have $T\subseteq F$, and hence $\F\subseteq \{F\in\M_{k}(k,n):T\subseteq F\}$.
		
On the other hand, the family $\{F\in\M_{k}(k,n):T\subseteq F\}$ is a $t$-star of $\M_{k}(k,n)$ with size $\frac{(n-t)!}{(n-k)!}$ and equals $|\F|$. This implies that $\F=\{F\in\M_{k}(k,n):T\subseteq F\}$.
\end{proof}

\section{Proof of Theorem \ref{thm:main-result} for $r=2$}\label{section_GP}

Similar to Section \ref{section_preliminaries}, an edge in $\mathcal{K}(n_1,n_2)$ is viewed as an ordered pair $(x,y)\in [n_1]\times[n_2]$. An element of $\M_{k}(n_1,n_2)$ is a set $F\subseteq [n_1]\times[n_2]$ of size $k$. A natural idea is to try to obtain a result similar to Lemma \ref{lem:fibers} using the concept of fibers. However, distinct $F_1,F_2\in \M_{k}(n_1,n_2)$ may have the same extension in $\M_{n_2}(n_2,n_2)$. For example, take $F_1=\{(1,1),(2,2)\}\in\M_{2}(3,4)$ and $F_2=\{(1,1),(3,3)\}\in\M_{2}(3,4)$. $F_1$ and $F_2$ are both contained in $P=\{(1,1),(2,2),(3,3),(4,4)\}$ $\in\M_{4}(4,4)$. Hence for distinct $F_1,F_2\in\F\subseteq\M_{k}(n_1,n_2)$ with $k\leq n_1\leq n_2$, the families $\{P\in \M_{n_2}(n_2,n_2):F_1\subseteq P\}$ and $\{P\in \M_{n_2}(n_2,n_2):F_2\subseteq P\}$ may not be disjoint. This leads to difficulties in counting. A modification is first to bridge $\M_{k}(n_1,n_2)$ and $\M_{n_1}(n_1,n_2)$, rather than $\M_{n_2}(n_2,n_2)$. For $F\in \M_{k}(n_1,n_2)$, we define
\begin{align*}
	\operatorname{Ext}_{(n_1,n_2)}(F) = \{G\in \M_{n_1}(n_1,n_2):F\subseteq G\}.
\end{align*}
For $\F\subseteq \M_{k}(n_1,n_2)$, define the family
\begin{align}\label{eq:def of mathcal_G}
	\mathcal{G} = \bigcup_{F\in\F}\operatorname{Ext}_{(n_1,n_2)}(F) \subseteq \M_{n_1}(n_1,n_2).
\end{align}
Furthermore, for every $G\in\mathcal{G}$, define
\begin{align}\label{eq:def_mathcal_F_G}
	\F_G=\{F\in \F:F\subseteq G\}\subseteq \F.
\end{align}

\begin{Lemma}\label{lem:relation_M_k and M_n_1}
Let $\F\subseteq\M_{k}(n_1,n_2)$ be a $t$-intersecting family with $1\leq t\leq k\leq n_1\leq n_2$. Let $\mathcal{G}$ and $\F_{G}$ be as in \eqref{eq:def of mathcal_G} and \eqref{eq:def_mathcal_F_G}, respectively. Then the following hold.
\begin{enumerate}
\item[\em(1)] $\mathcal{G}\subseteq\M_{n_1}(n_1,n_2)$ is $t$-intersecting.
\item[\em(2)] For any $G\in\mathcal{G}$, $\F_{G}\subseteq \F$ is $t$-intersecting, and $\F_{G}$ is isomorphic to a $t$-intersecting subfamily of $\binom{[n_1]}{k}$.
\item[\em(3)] $|\operatorname{Ext}_{(n_1,n_2)}(F)|=\frac{(n_2-k)!}{(n_2-n_1)!}$ for any $F\in\F$. Furthermore, $\sum_{G\in\mathcal{G}}|\F_G|=\frac{(n_2-k)!}{(n_2-n_1)!}|\F|$.
\end{enumerate}
\end{Lemma}

\begin{proof}
(1) and (2) hold immediately by the definitions of $\mathcal{G}$ and $\F_G$.

(3) Every $F\in\F$ uses $k$ distinct values of $[n_1]$ and $k$ distinct values of $[n_2]$. To extend $F$ to an element in $\M_{n_1}(n_1,n_2)$, it suffices to fill the remaining $n_1-k$ values in $[n_1]$ by the remaining $n_2-k$ values in $[n_2]$. The number of such choices equals $\binom{n_2-k}{n_1-k}\cdot(n_1-k)!=\frac{(n_2-k)!}{(n_2-n_1)!}$. Thus for any $F\in \F$, $|\operatorname{Ext}_{(n_1,n_2)}(F)|=\frac{(n_2-k)!}{(n_2-n_1)!}$.

Now let $\mathcal{I}(\F,\mathcal{G}) = \{(F,G):F\in\F,G\in\mathcal{G},F\subseteq G\}$. Counting the size of $\mathcal{I}(\F,\mathcal{G})$, we have
\begin{align*}
\sum_{G\in\mathcal{G}} |\F_G|
=|\mathcal{I}(\F,\mathcal{G})|
= \sum_{F\in\F}\left|\operatorname{Ext}_{(n_1,n_2)}(F)\right|
= \frac{(n_2-k)!}{(n_2-n_1)!}|\F|.
\end{align*}
This completes the proof.
\end{proof}

To give an upper bound for the size of a $t$-intersecting family $\F\subseteq\M_{k}(n_1,n_2)$, by Lemma \ref{lem:relation_M_k and M_n_1}(3), it suffices to estimate the size of $\F_G$ for every $G\in \mathcal G$. We need the concept of $t$-covers. A {\em $t$-cover} of $\F_G$ is a subset $S\subseteq [n_1]\times[n_2]$ such that $|S\cap F|\geq t$ for all $F\in \F_G$. The {\em $t$-covering number} $\tau(\F_G)$ of $\F_G$ is defined as the minimum size of $S$ such that $S$ is a $t$-cover of $\F_{G}$, i.e.,
	\begin{align}\label{eq:definition_covering number}
		\tau(\F_G) = \min_{S\subseteq[n_1]\times[n_2]}\{|S|:|S\cap F|\geq t~\text{for all $F\in\F_G$}\}.
	\end{align}
	Note that $\tau(\F_G)=t$ if and only if all elements of $\F_G$ contain a common $t$-subset of $[n_1]\times[n_2]$. In this setting, we can decompose $\mathcal{G}$ as (recall that ``$\sqcup$" represents the disjoint union)
\begin{align}\label{eq:decomposition of mathcal_G}
\mathcal{G}=\mathcal{G}_0 \sqcup \mathcal{G}_1,~\text{where}~
	\begin{cases}
		\mathcal{G}_0=\{G\in\mathcal{G}:\tau(\F_G)=t\};\\
		\mathcal{G}_1=\{G\in\mathcal{G}:\tau(\F_G)\geq t+1\}.
	\end{cases}
\end{align}
Write
$$N_1:=\binom{n_1-t}{k-t}\quad\text{and}\quad N_2:=\frac{(n_2-t)!}{(n_2-n_1)!}.$$

\begin{Lemma}\label{lem:sum_G leq N1N2}
Let $\F\subseteq\M_{k}(n_1,n_2)$ be a $t$-intersecting family with $1\leq t< k\leq n_1\leq n_2$ and $n_2\geq R(t)$, where $R(t)=\max\{C_0(t+1),t+4((t+1)^t-1)\}$. Let $\mathcal{G}$, $\F_{G}$ and $\mathcal{G}_1$ be as in (\ref{eq:def of mathcal_G}), (\ref{eq:def_mathcal_F_G}) and (\ref{eq:decomposition of mathcal_G}), respectively. Then the following hold.
\begin{enumerate}
\item[\em(1)] If $\mathcal{G}_1=\varnothing$, then $\sum_{G\in\mathcal{G}}|\F_G| \leq N_1 N_2$. Furthermore, if equality holds, then $|\mathcal{G}|=N_2$, and $|\F_G|=N_1$ for every $G\in\mathcal{G}$.
\item[\em(2)] If $\mathcal{G}_1\neq \varnothing$, then $\sum_{G\in\mathcal{G}}|\F_G| < N_1 N_2$. Moreover,
\begin{enumerate}
	\item[\em(2-1)] $|G^*\cap G|\geq t+1$ for any $G^*\in\mathcal{G}_1$ and any $G\in\mathcal{G}$;
	\item[\em(2-2)] $|\mathcal{G}_1|\leq \dfrac{(n_2-t-1)!}{(n_2-n_1)!}$ and $|\mathcal{G}|<0.75N_2$;
	\item[\em(2-3)] if $G\in\mathcal{G}_0$, then $|\F_G|\leq N_1$; if $G\in\mathcal{G}_1$, then $|\F_G|\leq (t+1)^tN_1$.
\end{enumerate}
\end{enumerate}
\end{Lemma}

To prove Lemma \ref{lem:sum_G leq N1N2}, we need the concept of derangements. Let $S$ be an $m$-element set. A {\em derangement} on $S$ is a permutation $\sigma:S\to S$ with no fixed points, i.e., $\sigma(s)\neq s$ for all $s\in S$. Denote the number of derangements on $S$ by $D_m$.
\begin{Lemma}\label{lem:derangements}
    If $m\geq 2$, then $D_m > 0.25m!$
\end{Lemma}

\begin{proof}
For $m=2$ and $m=3$, the assertion follows from $D_2=1$ and $D_3=2$, respectively. For $m\geq 4$, we use the recurrence $D_m=(m-1)(D_{m-1}+D_{m-2})$ (see \cite[Example 14.1]{vanlint_a course in comb}) and induction. Assume the lemma holds for $m-1$ and $m-2$ with $m\ge 4$. Then $D_{m-1}>0.25(m-1)!$ and $D_{m-2}>0.25(m-2)!$, so
\begin{align*}
 D_m &= (m-1)(D_{m-1}+D_{m-2}) > (m-1)(0.25(m-1)!+0.25(m-2)!)\\
    &= 0.25(m-1)((m-1)(m-2)!+(m-2)!) = 0.25m!.
\end{align*}
Thus $D_m > 0.25 m!$ for all $m\geq 2$.
\end{proof}

\begin{proof}[\bf Proof of Lemma \ref{lem:sum_G leq N1N2}]
(1) If $\mathcal{G}_1=\varnothing$, then $\mathcal{G}=\mathcal{G}_0$, and so for every $G\in\mathcal{G}$, all elements of $\F_G$ contain a common $t$-subset $S_G$. Hence $\F_G\subseteq \{F\subseteq G: S_G\subseteq F, |F|=k\}$, which implies $|\F_G| \leq \binom{n_1-t}{k-t} = N_1$ for every $G\in\mathcal{G}$. Thus, $\sum_{G\in\mathcal{G}}|\F_G|\leq |\mathcal{G}|N_1$. By Lemma \ref{lem:relation_M_k and M_n_1}(1), $\mathcal{G}\subseteq \M_{n_1}(n_1,n_2)$ is $t$-intersecting, so applying Theorem \ref{thm:main-result-injections}, we have $|\mathcal{G}|\leq \frac{(n_2-t)!}{(n_2-n_1)!}=N_2$. Therefore, $\sum_{G\in\mathcal{G}}|\F_G|\leq N_1 N_2$. Furthermore, if equality holds, then $|\mathcal{G}|=N_2$, and $|\F_G|=N_1$ for every $G\in\mathcal{G}$.

(2) In the following, we always assume that $\mathcal{G}_1\neq \varnothing$.
    			
(2-1) For any $G\in\mathcal{G}$, by (\ref{eq:def of mathcal_G}), there exists $F_0\in\F$ such that $F_0\subseteq G$. For any $G^*\in\mathcal{G}_1$, take an arbitrary $F^*\in\F_{G^*}\subseteq \F$. By (\ref{eq:def_mathcal_F_G}), $F^*\in\F$ and $F^*\subseteq G^*$. Since $\F$ is $t$-intersecting, $|F^*\cap F_0|\geq t$. Since $F^*\subseteq G^*$, $F^*\cap F_0=(F^*\cap G^*)\cap F_0=(G^*\cap F_0)\cap F^*$ and so $|(G^*\cap F_0)\cap F^*|=|F^*\cap F_0|\geq t$ for any $F^*\in\F_{G^*}$, which implies that $G^*\cap F_0$ is a $t$-cover of $\F_{G^*}$. Note that $\tau(\F_{G^*})\geq t+1$ because $G^*\in\mathcal{G}_1$. Hence for any $G\in\mathcal{G}$ and $G^*\in\mathcal{G}_1$,
\begin{align*}
    t+1 \leq \tau(\F_{G^*}) \leq |G^*\cap F_0|\leq |G^*\cap G|.
\end{align*}

(2-2) We first prove that $\mathcal{G}_1$ is $(t+1)$-intersecting. If $|\mathcal{G}_1|=1$, since $t<k$, $\mathcal{G}_1$ is $(t+1)$-intersecting. Suppose $|\mathcal{G}_1|\geq 2$. For any $G_1,G_2\in\mathcal{G}_1$, by (2-1), we have $|G_1\cap G_2|\geq t+1$, so $\mathcal{G}_1$ is $(t+1)$-intersecting. Since $\mathcal{G}_1\subseteq\M_{n_1}(n_1,n_2)$ and $n_2\geq R(t)\geq C_0(t+1)$, by Theorem \ref{thm:main-result-injections}, for the $(t+1)$-intersecting family $\mathcal{G}_1$, we have $|\mathcal{G}_1|\leq \dfrac{(n_2-t-1)!}{(n_2-n_1)!}$.
    			
Next we show that $|\mathcal{G}|< 0.75N_2$. Define a mapping
\begin{align*}
\gamma:~\M_{n_2}(n_2,n_2)\to \M_{n_1}(n_1,n_2),\ \ \text{where }\gamma(P)=P\cap ([n_1]\times[n_2])\ \ \text{for }P\in\M_{n_2}(n_2,n_2).
\end{align*}
By Lemma \ref{lem:relation_M_k and M_n_1}(1), $\mathcal{G}\subseteq \M_{n_1}(n_1,n_2)$ is $t$-intersecting. Then by Lemma \ref{lem:fibers}(2) and (3),
\begin{align*}
    \gamma^{-1}(\mathcal{G}) = \{P\in \M_{n_2}(n_2,n_2):\gamma(P)\in\mathcal{G}\}
\end{align*}
is also $t$-intersecting and $|\gamma^{-1}(\mathcal{G})|=(n_2-n_1)!|\mathcal{G}|$.
    			
If $\gamma^{-1}(\mathcal{G})$ is not contained in any $t$-star of $\M_{n_2}(n_2,n_2)$, then since $n_2\geq R(t)\geq C_0(t+1)>C_0t$, by Theorem \ref{thm:KLMS}(2) and its contrapositive, $|\gamma^{-1}(\mathcal{G})|<0.75(n_2-t)!$. Combining $N_2=\frac{(n_2-t)!}{(n_2-n_1)!}$ and $|\gamma^{-1}(\mathcal{G})|=(n_2-n_1)!|\mathcal{G}|$, we have $|\mathcal{G}|<0.75N_2$.
    			
If $\gamma^{-1}(\mathcal{G})$ is contained in some $t$-star of $\M_{n_2}(n_2,n_2)$, then suppose that
\begin{align}\label{eq:if mathcal_G is contained in a t-star}
    \gamma^{-1}(\mathcal{G}) \subseteq \mathcal{P}_{n_2}(T) := \{P\in\M_{n_2}(n_2,n_2):T\subseteq P\}
\end{align}
for a certain $T\in\M_{t}(n_2,n_2)$. Given any $G^*\in\mathcal{G}_1$, choose $P^*\in\gamma^{-1}(\mathcal{G})$ such that $G^*\subseteq P^*$. Then by (\ref{eq:if mathcal_G is contained in a t-star}),
$T\subseteq P^*$. Define
\begin{align}\label{eq:def of mathcal_D}
\mathcal{D}=\{P\in \mathcal{P}_{n_2}(T): P\cap (P^*\setminus T)=\varnothing\}.
\end{align}
We claim that
$$\gamma^{-1}(\mathcal{G})\subseteq \mathcal{P}_{n_2}(T)\setminus \mathcal{D}.$$
It suffices to show that $\gamma^{-1}(\mathcal{G})\cap \mathcal{D}=\varnothing$. Suppose not. Then there exists $P\in \gamma^{-1}(\mathcal{G})\cap \mathcal{D}$. Hence there is $G\in\mathcal{G}$ such that $\gamma(P)=G$ and $G\subseteq P$. By \eqref{eq:def of mathcal_D}, $P\cap (P^*\setminus T)=\varnothing$. By \eqref{eq:if mathcal_G is contained in a t-star}, $T\subseteq P$. Since $T\subseteq P^*$, it follows that $P\cap P^*=T$. Thus $|G\cap G^*|\leq |P\cap P^*|=|T|=t$, contradicting the fact proved in (2-1).
Therefore, $\gamma^{-1}(\mathcal G)\subseteq \mathcal P_{n_2}(T)\setminus \mathcal D$, which implies  that
$$
|\gamma^{-1}(\mathcal G)|
\le |\mathcal P_{n_2}(T)\setminus \mathcal D|
= |\mathcal P_{n_2}(T)| - |\mathcal D|.
$$
Clearly $|\mathcal P_{n_2}(T)| = (n_2-t)!$. Now we estimate the size of $\mathcal D$. Regard every $P\in\mathcal P_{n_2}(T)$ as a permutation on the remaining $n_2-t$ points after fixing the $t$-set $T$. The condition $P\cap(P^*\setminus T)=\varnothing$ means that this permutation has no fixed points relative to the fixed set $P^*\setminus T$. Hence $\mathcal D$ is in bijection with the set of derangements on $n_2-t$ points, so
$|\mathcal D| = D_{n_2-t}$.
Since $n_2\ge R(t)\ge t+2$ (as $R(t)\ge t+4((t+1)^t-1)\ge t+2$), we have $n_2-t\ge 2$.
By Lemma \ref{lem:derangements}, $|\mathcal D| = D_{n_2-t} > 0.25(n_2-t)!$.
Consequently,
\[
|\gamma^{-1}(\mathcal G)|
\le |\mathcal P_{n_2}(T)| - |\mathcal D|
< (n_2-t)! - 0.25(n_2-t)!
= 0.75(n_2-t)!.
\]
Combining this with $|\gamma^{-1}(\mathcal G)|=(n_2-n_1)!|\mathcal G|$ and $N_2=\frac{(n_2-t)!}{(n_2-n_1)!}$, we conclude that $|\mathcal G|<0.75N_2$.
    			
(2-3) If $G\in\mathcal{G}_0$, then $\tau(\F_{G})=t$ for any $G\in\mathcal{G}_0$. That is, every member of $\F_{G}$ contains a common $t$-subset, and so $|\F_{G}|\leq \binom{n_1-t}{k-t}=N_1$.

If $G\in\mathcal{G}_1$, by Lemma \ref{lem:relation_M_k and M_n_1}(2) and Theorem \ref{thm:classical EKR}, if $n_1\geq (t+1)(k-t+1)$, then $|\F_G|\leq \binom{n_1-t}{k-t}=N_1$. If $k\leq n_1<(t+1)(k-t+1)$, then for each $0\le i\le t-1$ we have $n_1-i \le n_1$ and $k-i \ge k-t+1$, hence
$$
\frac{n_1-i}{k-i} \le \frac{n_1}{k-t+1}.
$$
Therefore,
$$
\frac{\binom{n_1}{k}}{\binom{n_1-t}{k-t}} = \prod_{i=0}^{t-1} \frac{n_1-i}{k-i}
\le \left(\frac{n_1}{k-t+1}\right)^t < (t+1)^t,
$$
where the last inequality follows from $n_1<(t+1)(k-t+1)$. Since $|\F_G|\le \binom{n_1}{k}$, we obtain
$$
|\F_G|
\le \binom{n_1}{k}
= \frac{\binom{n_1}{k}}{\binom{n_1-t}{k-t}}\cdot \binom{n_1-t}{k-t}
< (t+1)^t N_1.
$$
Thus $|\F_G|\le (t+1)^tN_1$ for any $G\in\mathcal{G}_1$.
    			
Finally, since $n_2-t\geq R(t)-t\geq 4((t+1)^t-1)$, we have
\begin{align}\label{ineq:upper bound of mathcal_H neq varnothing}
\sum_{G\in\mathcal{G}}|\F_G| & =\sum_{G\in\mathcal{G}_0}|\F_G| + \sum_{G\in\mathcal{G}_1}|\F_G|\leq (|\mathcal{G}|-|\mathcal{G}_1|)N_1 + |\mathcal{G}_1|(t+1)^tN_1\\ \nonumber
    	    &=|\mathcal{G}|N_1 + |\mathcal{G}_1|((t+1)^t-1)N_1<0.75N_1N_2 + \frac{N_2}{n_2-t}((t+1)^t-1)N_1 \leq N_1N_2.
\end{align}
where the strict inequality follows from (2-2).
\end{proof}

Now we are ready to deal with the case of $r=2$ and $1\leq t\leq k\leq n_1\leq n_2$.

\begin{Lemma}[The case $r=2$ of Theorem \ref{thm:main-result}]\label{lem:result-generlized permutation}
Let $\F\subseteq \M_{k}(n_1,n_2)$ be a $t$-intersecting family with $1\leq t\leq k\leq n_1\leq n_2$. Let $C_0$ be a constant such that Theorem \ref{thm:KLMS} holds, and let $R(t)=\max\{C_0(t+1),t+4((t+1)^t-1)\}$. If $n_2\geq R(t)$, then $|\F|\leq \binom{n_1-t}{k-t}\binom{n_2-t}{k-t}(k-t)!$, and equality holds if and only if $\F$ is a $t$-star of $\M_{k}(n_1,n_2)$.
\end{Lemma}

\begin{proof}
The case of $t=k$ is trivial. The case of $r=2$ and $k=n_1$ follows from Theorem \ref{thm:main-result-injections}. Thus we may assume $1\leq t< k<n_1\leq n_2$.

By Lemma \ref{lem:relation_M_k and M_n_1}(3) and Lemma \ref{lem:sum_G leq N1N2}, we have
\begin{align*}
|\F| = \frac{(n_2-n_1)!\sum_{G\in\mathcal{G}}|\F_G|}{(n_2-k)!}\leq \frac{(n_2-n_1)!N_1N_2}{(n_2-k)!}
= \binom{n_1-t}{k-t}\binom{n_2-t}{k-t}(k-t)!.
\end{align*}

Suppose that $|\F|=\binom{n_1-t}{k-t}\binom{n_2-t}{k-t}(k-t)!$, which implies $\sum_{G\in\mathcal{G}}|\F_G|=N_1N_2$. By Lemma \ref{lem:sum_G leq N1N2}, this forces $\mathcal G_1=\varnothing$ and $|\mathcal G|=N_2$. Therefore, by Theorem \ref{thm:main-result-injections}, there is $T\in \M_{t}(n_1,n_2)$ such that
\begin{align*}
	\mathcal{G}=\{G\in\M_{n_1}(n_1,n_2):T\subset G\}.
\end{align*}
We claim that $\F = \{F\in\M_{k}(n_1,n_2):T\subseteq F\}$, which is a $t$-star of $\M_{k}(n_1,n_2)$.

First we prove that $\F \subseteq \{F\in\M_{k}(n_1,n_2):T\subseteq F\}$. Suppose, to the contrary, that there exists some $F^*\in\F$ with $T\not\subseteq F^*$. Then there exists $G^*\in \mathcal G$ such that $F^*\subset G^*$ (note that $k<n_1$). Since $T\subset G^*$, by Lemma \ref{lem:projection} with $r=2$, there exists $G'\in\M_{n_1}(n_1,n_2)$ such that $T\subset G'$ and $G^*\cap G'=T$. Since $T\subset G'$, $G'$ is also an element of $\mathcal{G}$. By (\ref{eq:def of mathcal_G}), there exists $F'\in\F$ such that $F'\subset G'$. Thus $F^*\cap F' \subseteq G^*\cap G' = T$. Clearly $F^*\cap F' \neq T$, since $T\not\subseteq F^*$. It follows that $|F^*\cap F'|<|G^*\cap G'|=|T|=t$, contradicting the $t$-intersecting property of $\F$. Hence $\F \subseteq \{F\in\M_{k}(n_1,n_2):T\subseteq F\}$.
		
On the other hand, the family $\{F\in\M_{k}(n_1,n_2):T\subseteq F\}$ is a $t$-star of $\M_{k}(n_1,n_2)$ whose size is $\binom{n_1-t}{k-t}\binom{n_2-t}{k-t}(k-t)! = |\F|$. Combining this with $\F \subseteq \{F\in\M_{k}(n_1,n_2):T\subseteq F\}$, we have $\F = \{F\in\M_{k}(n_1,n_2):T\subseteq F\}$.
\end{proof}

\section{Proof of Theorem \ref{thm:main-result} for $r\geq 3$}\label{section_HP}
	
For $r\geq 3$, an edge of $\mathcal{K}(n_1,n_2,\dots,n_r)$ is viewed as an ordered $r$-tuple $(x_1,x_2,\dots,x_r)\in [n_1]\times[n_2]\times\dots\times [n_r]$. An element $F\in \M_{k}(n_1,n_2,\dots,n_r)$ is a set of $k$ edges of the form $F = \{{\bm x}^{(1)},\dots,{\bm x}^{(k)}\} \subseteq [n_1]\times[n_2]\times\dots\times [n_r]$, where ${\bm x}^{(j)}=(x_1^{(j)},x_2^{(j)},\dots,x_r^{(j)})$ for $j\in [k]$, and for each $i\in[r]$, the coordinates $x_i^{(1)},\dots,x_i^{(k)}$ are pairwise distinct. For an edge ${\bm x}=(x_1,\dots,x_{r-1},x_r)$ of $\mathcal{K}(n_1,n_2,\dots,n_r)$, define its projection to the first $r-1$ coordinates by
\begin{align*}
	\pi({\bm x}) = (x_1,\dots,x_{r-1}).
\end{align*}
For $F\in \M_{k}(n_1,n_2,\dots,n_r)$, define $\pi(F)=\{\pi({\bm x}):{\bm x}\in F\}\in \M_{k}(n_1,\dots,n_{r-1})$. For $\F\subseteq \M_{k}(n_1,n_2,\dots,n_r)$, define $\pi(\F)=\{\pi(F):F\in\F\}\subseteq \M_{k}(n_1,\dots,n_{r-1})$.

\begin{Lemma}\label{lem:fibers_multipartite}
Let $\F\subseteq \M_{k}(n_1,n_2,\dots,n_r)$ be a $t$-intersecting family with $r\geq 3$ and $1\leq t\leq k\leq n_1\leq n_2\leq \dots\leq n_r$. Then the following hold.
\begin{enumerate}
\item[\em(1)] $\pi(\F)$ is $t$-intersecting.
\item[\em(2)] For any $Q\in\pi(\F)$, let $\F_Q=\{F\in\F:\pi(F)=Q\}$. Then
\begin{enumerate}
\item[\em(2-1)] $\F=\bigsqcup_{Q\in\pi(\F)}\F_{Q}$, where $\bigsqcup$ represents the disjoint union;
\item[\em(2-2)] $\F_{Q}$ is isomorphic to a $t$-intersecting subfamily of $\M_{k}(k,n_r)$. Furthermore, if $n_r\geq C_0t$, where $C_0$ is a constant such that Theorem \ref{thm:KLMS} holds, then $|\F_{Q}|\leq \frac{(n_r-t)!}{(n_r-k)!}$.
\end{enumerate}
\end{enumerate}
\end{Lemma}

\begin{proof}
(1) and (2-1) hold immediately by the definitions of $\pi(\F)$ and $\F_Q$.
		
(2-2) For every $Q\in \pi(\F)$, since $\F_Q\subseteq\F$, $\F_Q$ is also $t$-intersecting. Each $F\in\F_Q$ can be viewed as a matching between the $k$ edges of $Q$ and $k$ distinct vertices of the last part $[n_r]$. Thus $\F_{Q}$ is naturally identified with a certain $t$-intersecting subfamily of $\M_{k}(k,n_r)$. Furthermore, if $n_r\geq C_0t$, then by Theorem \ref{thm:main-result-injections}, $|\F_{Q}|\leq \frac{(n_r-t)!}{(n_r-k)!}$.
\end{proof}
	
Now we give a proof of Theorem \ref{thm:main-result} for any $r\geq 2$.

	
\begin{proof}[\bf Proof of Theorem \ref{thm:main-result}]
We use induction on $r$. Let $\text{Stat}(k,r)$ denote the statement of the theorem for $\M_{k}(n_1,n_2$, $\dots,n_r)$. The case $r=2$ holds by Lemma~\ref{lem:result-generlized permutation}. Now let $r\geq 3$, and assume that $\text{Stat}(k,r-1)$ holds. We will prove that $\text{Stat}(k,r)$ also holds.
		
Let $\F\subseteq \M_{k}(n_1,n_2,\dots,n_r)$ be a $t$-intersecting family with $1\leq t\leq k\leq n_1\leq n_2\leq \dots \leq n_{r-1}\leq n_r$ and $n_2\geq R(t)$. The case of $t=k$ is trivial. So we assume that $1\leq t<k$. By Lemma \ref{lem:fibers_multipartite}(1), $\pi(\F)$ is also $t$-intersecting. By the induction hypothesis, we have
\begin{align}\label{ineq:upper bound of pi_mathcal_F}
	|\pi(\F)| \leq\left((k-t)!\right)^{r-2}\prod_{i=1}^{r-1}\binom{n_i-t}{k-t}.
\end{align}
Combining Lemma \ref{lem:fibers_multipartite}(2) and \eqref{ineq:upper bound of pi_mathcal_F}, we have
\begin{align*}
|\F| = \left|\bigsqcup_{Q\in\pi(\F)}\F_{Q}\right|
	= \sum_{Q\in \pi(\F)} |\F_Q|\leq |\pi(\F)|(k-t)!\binom{n_r-t}{k-t}
	\leq \left((k-t)!\right)^{r-1}\prod_{i=1}^{r}\binom{n_i-t}{k-t}.
\end{align*}

Now suppose that $|\F|=\left((k-t)!\right)^{r-1}\prod_{i=1}^{r}\binom{n_i-t}{k-t}$. Then $|\pi(\F)| = \left((k-t)!\right)^{r-2}\prod_{i=1}^{r-1}\binom{n_i-t}{k-t}$. By the induction hypothesis, there exists a fixed $T\in\M_{t}(n_1,\dots,n_{r-1})$ such that
\begin{align}\label{eq:structure of pi_F}
	\pi(\F)=\{Q\in\M_{k}(n_1,\dots,n_{r-1}):T\subseteq Q\}.
\end{align}
Moreover, equality holds for every $\F_Q$, i.e., for every $Q\in\pi(\F)$, $|\F_Q|=\frac{(n_r-t)!}{(n_r-k)!}$. For each such $Q$, by Lemma \ref{lem:fibers_multipartite}(2-2), $\F_Q$ is isomorphic to a $t$-intersecting subfamily of $\M_k(k,n_r)$. Since its size attains the extremal value in Theorem \ref{thm:main-result-injections} (note that $n_r\geq n_2\geq C_0t$), the equality case in Theorem \ref{thm:main-result-injections} implies that this isomorphic subfamily is a $t$-star of $\M_k(k,n_r)$.
Translating this star back via the isomorphism yields a set $C_Q\in \M_t(n_1,\dots,n_{r-1},n_r)$ such that $\pi(C_Q)\subseteq Q$ and
\begin{align}\label{eq:structure of F_Q}
    \F_Q = \{F\in \M_k(n_1,\dots,n_{r-1},n_r): \pi(F)=Q,\; C_Q\subseteq F\}.
\end{align}
We will show that $C_Q$ is the same for all $Q\in \pi(\F)$ by two claims.

\textbf{Claim (i).} $\pi(C_Q)=T$ for every $Q\in\pi(\F)$.
		
Suppose not. Then there exist $Q'\in\pi(\F)$ and an edge ${\bm e}\in T\setminus \pi(C_{Q'})$. Since $n_{r-1}\ge n_2\ge R(t)\ge t+2$, we have $n_{r-1}-t\ge 2$. Since $Q'\in\pi(\F)$, we have $T\subseteq Q'$ by \eqref{eq:structure of pi_F}. Applying Lemma \ref{lem:projection} to $\M_{k}(n_1,\dots,n_{r-1})$, we obtain $Q^*\in \M_{k}(n_1,\dots,n_{r-1})$ such that $T\subseteq Q^*$ and $Q'\cap Q^*=T$. Since $T\subseteq Q^*$, by (\ref{eq:structure of pi_F}), we have $Q^*\in\pi(\F)$. Choose $F^*\in\F_{Q^*}$. Since $\pi(F^*)=Q^*$ and ${\bm e}\in T\subseteq Q^*$, there is a unique edge in $F^*$ whose image under $\pi$ is ${\bm e}$; denote it by $({\bm e},z^*)$.

Let $D = Q' \setminus \pi(C_{Q'})$, and let $Y$ be the set of last coordinates not used by any edge of $C_{Q'}$. Then $|D|=k-t$ and $|Y|=n_r-t$, so $|D|\le |Y|$ and $|Y|\ge 2$ (since $n_r\ge n_2\ge t+2$). Let $X = \{{\bm d}\in D : ({\bm d}, z^*) \in F^*\}$. Since $F^*$ is a matching, $|X|\leq 1$, so the map $\psi:X\to Y$, with $\psi({\bm d})=z^*$, is injective. By Lemma \ref{lem:auxiliary_1}, there exists an injection $\xi:D\to Y$ such that $\xi({\bm d})\neq z^*$ for ${\bm d}\in X$.
Now set
$$
F' = C_{Q'} \cup \{({\bm d},\xi({\bm d})) : {\bm d}\in D\}.
$$
Then $\pi(F')=Q'$ and $C_{Q'}\subseteq F'$, so by (\ref{eq:structure of F_Q}) we have $F'\in\F_{Q'}$.
Moreover, $F'$ contains no edge of the form $({\bm e},z^*)$: indeed, $({\bm e},z^*)\notin C_{Q'}$ because ${\bm e}\notin \pi(C_{Q'})$, and for any ${\bm d}\in D$ with $\pi(({\bm d},\xi({\bm d})))={\bm e}$, the definition of $\xi$ ensures $\xi({\bm d})\neq z^*$.

Now $\pi(F')=Q'$ and $\pi(F^*)=Q^*$. Hence any common edge of $F^*$ and $F'$ must have its $\pi$-image in $Q'\cap Q^*=T$.
However, the edge $({\bm e},z^*)$ belongs to $F^*$ but not to $F'$, so no common edge has $\pi$-image equal to ${\bm e}$.
Therefore, $
|F^*\cap F'| \le |T\setminus\{{\bm e}\}| = t-1$, contradicting the $t$-intersecting property of $\F$. Thus Claim (i) holds.

\textbf{Claim (ii).} $C_Q$ is the same for all $Q\in \pi(\F)$.

Suppose not. Then there exist distinct $Q,Q'\in\pi(\F)$ such that $C_{Q}\neq C_{Q'}$. By Claim (i), $\pi(C_Q)=\pi(C_{Q'})=T$. Since $\pi$ only deletes the $r$-th coordinate, the two $t$-sets $C_{Q}$ and $C_{Q'}$ differ only in their last coordinates. Hence there exists ${\bm e_0}=(x_1,\dots,x_{r-1})\in T$ and distinct $y,y'\in [n_r]$ such that $({\bm e}_0,y)\in C_{Q}$ and $({\bm e}_0,y')\in C_{Q'}$.
		
Choose $F\in\F_{Q}$. We will construct $F'\in \F_{Q'}$ such that $(F\cap F')\setminus \pi^{-1}(T)=\varnothing$, where $\pi^{-1}(T)=\{({\bm e},z): {\bm e} \in T,z\in[n_r]\}$. Put
$$
D=Q'\setminus T,\quad
V_r(C_{Q'})=\{z\in[n_r]:\exists\,{\bm e}\in T \text{ with } ({\bm e},z)\in C_{Q'}\},\quad
Y=[n_r]\setminus V_r(C_{Q'}).
$$
Then $|D|=k-t$ and $|Y|=n_r-t$, so $|D|\le |Y|$ and $|Y|\ge 2$. For each ${\bm e}\in (Q\cap Q')\setminus T$, let $({\bm e},z_e)\in F$ be the unique edge with $\pi(({\bm e},z_e))={\bm e}$.
Define
$$
X=\{{\bm e}\in (Q\cap Q')\setminus T : z_e\in Y\},
$$
and let $\psi:X\to Y$ be given by $\psi({\bm e})=z_e$. Since $F$ is a matching, different ${\bm e}$'s have distinct $z_e$'s, so $\psi$ is injective.
By Lemma \ref{lem:auxiliary_1}, there exists an injection $\xi:D\to Y$ such that
$\xi({\bm e})\neq \psi({\bm e})$ for all ${\bm e}\in X$. Now set
$$
F' = C_{Q'} \cup \{({\bm e},\xi({\bm e})):{\bm e}\in D\}.
$$
Then $\pi(F')=Q'$ and $C_{Q'}\subseteq F'$, so by \eqref{eq:structure of F_Q} we have $F'\in\F_{Q'}$. We claim that $(F\cap F')\setminus \pi^{-1}(T)=\varnothing$.
Suppose not, and take ${\bm d}\in (F\cap F')\setminus \pi^{-1}(T)$.
Write $\pi({\bm d})={\bm e'}\notin T$. Then ${\bm e'}\in (Q\cap Q')\setminus T$ and ${\bm d}=({\bm e'},z_{e'})$. If $z_{e'}\notin Y$, then $F'$ cannot contain $({\bm e'},z_{e'})$ because all last coordinates in $F'\setminus C_{Q'}$ are chosen from $Y$.
If $z_{e'}\in Y$, then ${\bm e'}\in X$, and the construction of $\xi$ gives $\xi({\bm e'})\neq z_{e'}$, so again $({\bm e'},z_{e'})\notin F'$.
Both cases contradict ${\bm d}\in F'$. Hence $(F\cap F')\setminus \pi^{-1}(T)=\varnothing$.

Now, since $C_Q\subseteq F$ and $C_{Q'}\subseteq F'$, the edge $({\bm e}_0,y)$ lies in $F$ while $({\bm e}_0,y')$ lies in $F'$, with $y\ne y'$. Hence no edge of $F\cap F'$ can have $\pi$-image equal to ${\bm e}_0$.
Combined with $(F\cap F')\setminus \pi^{-1}(T)=\varnothing$, this means every common edge of $F$ and $F'$ must have its $\pi$-image in $T\setminus\{{\bm e}_0\}$.
Since $|T\setminus\{{\bm e}_0\}|=t-1$, we obtain $|F\cap F'| \le t-1$, contradicting the $t$-intersecting property of $\F$. Therefore $C_{Q}=C_{Q'}$ for all $Q,Q'\in\pi(\F)$. Thus Claim (ii) holds.

It remains to prove that $\F$ is a $t$-star of $\M_k(n_1,\dots,n_r)$.
Since $C_Q$ is the same for all $Q\in\pi(\F)$, there exists a fixed $C^*\in \M_t(n_1,\dots,n_r)$ such that $C_Q=C^*$ for every $Q\in\pi(\F)$. For any $F\in \F$, $F\in \F_{\pi(F)}$, and so $C^*\subseteq F$. Hence $\F\subseteq\{F\in\M_{k}(n_1,\dots,n_r):C^*\subseteq F\}$. For the reverse inclusion, take any $F'\in \M_k(n_1,\dots,n_r)$ with $C^*\subseteq F'$.
Since $\pi(C^*)=T$, we have $T\subseteq \pi(F')$, so by (\ref{eq:structure of pi_F}), $\pi(F')\in\pi(\F)$. By \eqref{eq:structure of F_Q}, $\F_{\pi(F')} = \{F\in \M_k(n_1,\dots,n_r): \pi(F)=\pi(F'),\ C^*\subseteq F\}$. The right-hand side contains $F'$ because $\pi(F')=\pi(F')$ and $C^*\subseteq F'$. Hence $F'\in\F_{\pi(F')}\subseteq\F$. Thus $\{F\in \M_k(n_1,\dots,n_r): C^*\subseteq F\} \subseteq \F$. This completes the proof.
\end{proof}

\section{Proof of Theorems \ref{main_result_1}, \ref{main_result_2} and \ref{main_result_3}} \label{section_t-covers}

The notion of $t$-covers, introduced in (\ref{eq:definition_covering number}), is a standard tool in extremal set theory. In this section we give a unified proof of Theorems \ref{main_result_1}--\ref{main_result_3} using this notion.

We now turn to a general form of the $t$-cover argument. Although the notion was already used in the proof of Lemma \ref{lem:sum_G leq N1N2} for the specific families $\F_G$, we restate it here in a broader setting. Let $\Omega$ be a nonempty set and $t$ be a positive integer. A subset $T$ of $\Omega$ is called a \emph{$t$-cover} (or \emph{$t$-transversal}) of a family $\F\subseteq 2^{\Omega}$ if $|F\cap T|\geq t$ for any $F\in\F$.  Define the {\em $t$-covering number} $C_t(\F)$ of a family $\F\subseteq 2^{\O}$ to be the size of minimum $t$-covers among all $t$-covers of $\F$.
	
For $1\leq t\leq k\leq |\O|$, the $t$-covering number of any $t$-intersecting family $\mathcal{F} \subseteq \binom{\Omega}{k}$ satisfies $t \leq C_t(\mathcal{F}) \leq k$, and $C_t(\mathcal{F})=t$ if and only if $\mathcal{F}$ is a $t$-star. Consequently, to prove that $t$-stars are the unique maximum-sized $t$-intersecting families, it suffices to show that any family $\mathcal{F}_1$ with a larger $t$-cover number ($C_t(\mathcal{F}_1) \geq t+1$) must satisfy $|\mathcal{F}_1| < |\mathcal{F}_0|$ for some $t$-star $\mathcal{F}_0$.
	
For a family $\F\subseteq\binom{\O}{k}$ and a subset $S$ of $\O$, define $\F_S=\{F\in\F: S\subseteq F\}$. The following lemma comes from \cite[Lemma 3.1]{cao2021structure}. We provide its proof for completeness.
	
\begin{Lemma}\label{key lemma} {\rm \cite[Lemma 3.1]{cao2021structure}}
Let $\O$ be a nonempty set and $1\leq t\leq k\leq |\O|$. Let $\mathcal{F}\subseteq\binom{\O}{k}$ be a $t$-intersecting family, and $H\in \binom{\O}{s}$ with $1\leq s<|\O|$. If there exists $F\in\mathcal{F}$ such that $|F\cap H|=\alpha< t$, then for each $1\leq \ell\leq t-\alpha$, there exists $H^*\in \binom{\O}{s+\ell}$ such that $H\subseteq H^*$ and $|\mathcal{F}_{H}|\leq \binom{k-\alpha}{\ell}|\mathcal{F}_{H^*}|$.
\end{Lemma}
	
	\begin{proof}
		The case $\F_H=\emptyset$ is trivial. Assume that $\F_H\neq \emptyset$. For any $F'\in\F_H$, since $\F$ is $t$-intersecting, we have $|F'\cap F|\geq t$. It follows that $|F'\cap (H\cup F)|\geq |F'\cap H|+|F'\cap F|-|H\cap F|\geq s+t-\alpha$. Consequently, for each $1\leq \ell\leq t-\alpha$, $\F_H=\bigcup_{L\in \binom{F\setminus H}{\ell}}\F_{H\cup L}$, and hence $|\F_H|\leq \binom{k-\alpha}{\ell} |\F_{H\cup L}|$ for some $L\in \binom{F\setminus H}{\ell}$. Take $H^*=H\cup L$ to complete the proof.
	\end{proof}

	\begin{Lemma}\label{lemma:general recursion}
		Let $\O$ be a nonempty set and $1\leq t<k<|\O|$. Let $\F\subseteq \binom{\O}{k}$ be a $t$-intersecting family with $t+1\leq C_t(\F)=c\leq k$. Then there exist $\omega$ and $H'\in\binom{\O}{\omega}$ with $c\leq \omega\leq k$ such that $\F_{H'}\neq \emptyset$ and
		\begin{align*}
			|\F|\leq \binom{c}{t}(k-t+1)^{\omega-t}|\F_{H'}|.
		\end{align*}
	\end{Lemma}
	
	\begin{proof}
		Since $C_t(\F)=c$, we can take a $t$-cover $T\in \binom{\O}{c}$ of $\F$ such that $|F\cap T|\geq t$ for any $F\in\F$. It follows that $\F=\bigcup_{H\in\binom{T}{t}}\F_H$ and so
		\begin{align}\label{ineq:sum_combin_bound_F and F_H_1}
			|\F|\leq \binom{c}{t}|\F_{H_1}|
		\end{align}
		for a certain $H_1\in\binom{T}{t}$. Since $|H_1|=t<c$, $H_1$ is not a $t$-cover of $\F$, and so there exists $F_1\in\F$ such that $|F_1\cap H_1|=\alpha_1<t$. By Lemma \ref{key lemma} with $H=H_1$ and $\ell=t-\alpha_1$, there exists $H_2\in\binom{\O}{|H_1|+(t-\alpha_1)}$ such that
		$H_1\subseteq H_2$ and
		\begin{align*}
			|\F_{H_1}|\leq \binom{k-\alpha_1}{t-\alpha_1}|\F_{H_2}|
			=\prod_{j=0}^{t-\alpha_1-1}\frac{k-\alpha_1-j}{t-\alpha_1-j} |\F_{H_2}|
			\leq (k-t+1)^{t-\alpha_1}|\F_{H_2}|.
		\end{align*}
		The above process can be iterated until $|H_m|\geq c$ for some $m\geq 2$. Specifically, we can construct a sequence $H_1\subseteq\cdots\subseteq H_{m-1}\subseteq H_m$ with $m\geq 2$ and $|H_{m-1}|<c\leq |H_m|$ such that for each $i$ with $1\leq i\leq m-1$, there exists $H_{i+1}\in\binom{\O}{|H_i|+t-\alpha_i}$ satisfying $H_i\subseteq H_{i+1}$ and
		\begin{align}\label{ineq:F_H_i and F_ H_i+1}
			|\F_{H_i}|\leq \binom{k-\alpha_i}{t-\alpha_i}|\F_{H_{i+1}}|\leq (k-t+1)^{t-\alpha_i}|\F_{H_{i+1}}|.
		\end{align}
		Note that $|H_{i+1}|-|H_i|=t-\alpha_i$, which yields $\sum_{i=1}^{m-1}(t-\alpha_i)=|H_m|-|H_1|=|H_m|-t$. Therefore, applying \eqref{ineq:F_H_i and F_ H_i+1} repeatedly, we have
		\begin{align}\label{ineq:F_H_1 and F_H_m}
			|\F_{H_1}|\leq (k-t+1)^{\sum_{i=1}^{m-1}(t-\alpha_i)}|\F_{H_m}|= (k-t+1)^{|H_m|-t}|\F_{H_m}|.
		\end{align}
		Combining (\ref{ineq:sum_combin_bound_F and F_H_1}) and $(\ref{ineq:F_H_1 and F_H_m})$, we obtain
		\begin{align}\label{ineq:20260102}
			|\F|\leq \binom{c}{t}(k-t+1)^{|H_m|-t}|\F_{H_m}|.
		\end{align}
		Finally, it remains to show $|H_m| \leq k$; then taking $H' = H_m$ completes the proof. Indeed,
		since  $c\geq t+1>0$, we have $\F\neq \emptyset$. It follows from \eqref{ineq:20260102} that $\F_{H_m}\neq \emptyset$. Therefore, there exists $F\in\F$ such that $H_m\subseteq F$, which implies $|H_m|\leq |F|=k$.
	\end{proof}
	
	The following lemma is an application of Lemma \ref{lemma:general recursion} that is specific to the EKR-type problem for hypergraph matchings.
	
	\begin{Lemma}\label{lemma:F_c for matchings}
		Let $r\geq 2$ and $1\leq t< k\leq \min_{1\leq i\leq r} n_i$.
		Let $\F\subseteq\M_k(n_1,\dots,n_r)$ be a $t$-intersecting family with $t+1\leq C_t(\F)=c\leq k$. If either $r=2$, $n_1=k$ and $n_2>2k-t$, or $r\geq 2$ and $\min_{1\leq i\leq r} n_i> ((k-t+1)(k-t))^{1/r} +k-1$, then
		\begin{align*}
			|\F|\leq \binom{c}{t}(k-t+1)^{c-t}\left((k-c)!\right)^{r-1}\prod_{i=1}^{r}\binom{n_i-c}{k-c}.
		\end{align*}
	\end{Lemma}
	
	\begin{proof}
		By Lemma \ref{lemma:general recursion} with $\O=E(\mathcal{G}_1(n_1,\dots,n_r))$, there exist $\omega$ and $H'\in\M_{\omega}(n_1,\dots,n_r)$ with $c\leq \omega \leq k$ such that
		\begin{align*}
			|\F|\leq \binom{c}{t}(k-t+1)^{\omega-t}|\F_{H'}|\leq \binom{c}{t}(k-t+1)^{\omega-t}
			\left((k-\omega)!\right)^{r-1}\prod_{i=1}^{r}\binom{n_i-\omega}{k-\omega},
		\end{align*}
		where the last inequality follows from \eqref{eq:20251205}. Define a function
		\begin{align*}	L(\omega)=(k-t+1)^{\omega-t}\left((k-\omega)!\right)^{r-1}\prod_{i=1}^{r}\binom{n_i-\omega}{k-\omega},
		\end{align*}
		where $c\leq \omega \leq k$. Then $|\F|\leq \binom{c}{t} L(w)$. To complete the proof, it suffices to show that $L(\omega)\leq L(c)$ for any $c\leq \omega \leq k$.
		
		If $c=k$, then $L(\omega)=L(c)$. Assume that $c<\omega\leq k$. For $r=2$, $n_1=k$ and $n_2>2k-t$, we have that
		\begin{align*}
			\frac{L(\omega)}{L(\omega-1)}=\frac{k-t+1}{n_2-\omega+1}\leq \frac{k-t+1}{n_2-k+1}<1.
		\end{align*}
		For $r\geq 2$ and $\min_{1\leq i\leq r} n_i> ((k-t+1)(k-t))^{1/r} +k-1$, we have
		\begin{align*}
			\frac{L(\omega)}{L(\omega-1)}
			=\frac{(k-t+1)(k-\omega+1)}{\prod_{i=1}^{r}(n_i-\omega+1)}
			<\frac{(k-t+1)(k-c+1)}{(\min_{1\leq i\leq r} n_i-k+1)^r}
			\leq \frac{(k-t+1)(k-t)}{(\min_{1\leq i\leq r} n_i-k+1)^r}
			<1.
		\end{align*}
		Therefore, $L(\omega)$ is decreasing on the interval $[c,k]$.
	\end{proof}
	
	\subsection{Proof of Theorem \ref{main_result_1}}
	\begin{Lemma}\label{lemma:nontrivial t-intersecting families of injections}
		Let $1\leq t< k$ and $n>(t+1)(k-t+1)+t$. Let $\F\subseteq\M_k(k,n)$ be a $t$-intersecting family with $t+1\leq C_t(\F)=c\leq k$. Then
		\begin{align*}
			|\F|\leq (t+1)(k-t+1)\frac{(n-t-1)!}{(n-k)!}.
		\end{align*}
	\end{Lemma}
	
	\begin{proof}
		Apply Lemma \ref{lemma:F_c for matchings} with $r=2$, $n_1=k$ and $n_2=n$, where $n>(t+1)(k-t+1)+t>2k-t$. Then we have
		\begin{align*}
			|\F|\leq \binom{c}{t}(k-t+1)^{c-t}\frac{(n-c)!}{(n-k)!}.
		\end{align*}
		Define a function
		\begin{align*}
			L_1(c)=\binom{c}{t}(k-t+1)^{c-t}\frac{(n-c)!}{(n-k)!},
		\end{align*}
		where $t+1\leq c\leq k$. Then $|\F|\leq L_1(c)$. To complete the proof, it suffices to show that $L_1(c)\leq L_1(t+1)$ for any $t+1\leq c\leq k$.
		
		If $k=t+1$, then $c=t+1$, and so $L_1(c)=L_1(t+1)$. Assume that $t+1<c\leq k$. Then
		\begin{align*}
			\frac{L_1(c-1)}{L_1(c)}=\frac{(n-c+1)(c-t)}{(k-t+1)c}> \frac{((t+1)(k-t+1)+t-c+1)(c-t)}{(k-t+1)c}\geq 1.
		\end{align*}
		Therefore, $L_1(c)$ is decreasing on the interval $[t+1,k]$.
	\end{proof}
	
	\begin{proof}[\bf Proof of Theorem \ref{main_result_1}.] Let $\F_0$ be a $t$-star of $\M_k(k,n)$. By \eqref{eq:20251205} and Lemma \ref{lemma:nontrivial t-intersecting families of injections}, it suffices to show that
		\begin{align*}
			|\F_0|=\frac{(n-t)!}{(n-k)!}>(t+1)(k-t+1)\frac{(n-t-1)!}{(n-k)!}.
		\end{align*}
		This holds for any $n>(t+1)(k-t+1)+t$.
	\end{proof}

	\subsection{Proof of Theorem \ref{main_result_2}}
	
	\begin{Lemma}\label{lemma:nontrivial t-intersecting families of generalized permutations}
		Let $1<t<k$ and $\min\{n_1,n_2\}\geq \sqrt{(k-t+1)(k-t)(t+2)}+t$. Let $\F\subseteq\M_k(n_1,n_2)$ be a $t$-intersecting family with $t+1\leq C_t(\F)=c\leq k$. Then
		\begin{align*}
			|\F|\leq (t+1)(k-t+1)\binom{n_1-t-1}{k-t-1}\binom{n_2-t-1}{k-t-1}(k-t-1)!.
		\end{align*}
	\end{Lemma}
	
	\begin{proof}
		Since $t\geq 2$, we have
		\begin{align*}
			\min\{n_1,n_2\} & \geq \sqrt{(k-t+1)(k-t)(t+2)}+t   \\
			& \geq 2\sqrt{(k-t+1)(k-t)}+t >\sqrt{(k-t+1)(k-t)} +k-1,
		\end{align*}
		and so we can apply Lemma \ref{lemma:F_c for matchings} with $r=2$ to obtain
		\begin{align*}
			|\F|\leq \binom{c}{t}(k-t+1)^{c-t}\binom{n_1-c}{k-c}\binom{n_2-c}{k-c}(k-c)!.
		\end{align*}
		Define a function
		\begin{align*}
			L_2(c)=\binom{c}{t}(k-t+1)^{c-t}\binom{n_1-c}{k-c}\binom{n_2-c}{k-c}(k-c)!,
		\end{align*}
		where $t+1\leq c\leq k$. Then $|\F|\leq L_2(c)$. To complete the proof, it suffices to show that $L_2(c)\leq L_2(t+1)$ for any $t+1\leq c\leq k$.
		
		If $k=t+1$, then $c=t+1$ and $L_2(c)=L_2(t+1)$. Assume that $t+1<c\leq k$. Then
		\begin{align*}
			\frac{L_2(c-1)}{L_2(c)}&=\frac{(n_1-c+1)(n_2-c+1)(c-t)}{(k-t+1)(k-c+1)c}\geq \frac{(\min\{n_1,n_2\}-c+1)^2(c-t)}{(k-t+1)(k-c+1)c}\\
			& \geq \frac{(\sqrt{(k-t+1)(k-t)(t+2)}+t-c+1)^2(c-t)}{(k-t+1)(k-c+1)c}.
		\end{align*}
		To prove that $L_2(c-1)>L_2(c)$, we define the auxiliary function
		$$L_3(c)=(\sqrt{(k-t+1)(k-t)(t+2)}+t-c+1)^2(c-t)-(k-t+1)(k-c+1)c,$$
		where $t+1<c\leq k$. It is readily checked that $L_3(c)$ is increasing on the interval $[t+2,k]$ by analyzing the monotonicity of its derivative function, and so $L_3(c)\geq L_3(t+2)>0$ for $t+2\leq c \leq k$. Therefore, $L_2(c-1)>L_2(c)$, and hence $L_2(c)$ is decreasing on the interval $[t+1,k]$.
	\end{proof}
	
	\begin{proof}[\bf Proof of Theorem \ref{main_result_2}.]
		Let $\F_{0}$ be a $t$-star of $\M_k(n_1,n_2)$. By \eqref{eq:20251205} and Lemma \ref{lemma:nontrivial t-intersecting families of generalized permutations}, it suffices to show that
		\begin{align*}	|\F_{0}|=\binom{n_1-t}{k-t}\binom{n_2-t}{k-t}(k-t)!>(t+1)(k-t+1)
			\binom{n_1-t-1}{k-t-1}\binom{n_2-t-1}{k-t-1}(k-t-1)!,
		\end{align*}
		i.e., to show that $(n_1-t)(n_2-t)>(k-t+1)(k-t)(t+1)$.
		This holds for $\min\{n_1,n_2\}\geq \sqrt{(k-t+1)(k-t)(t+2)}+t> \sqrt{(k-t+1)(k-t)(t+1)}+t$.
	\end{proof}

	\subsection{Proof of Theorem \ref{main_result_3}}
	
	\begin{Lemma}\label{lemma_proof_3}
		Let $1\leq t< k <\min_{1\leq i\leq r} n_i$ and $r\geq \log_2((t+1)(k-t+1)^2)$. Let $\F\subseteq\M_k(n_1,\dots,n_r)$ be a $t$-intersecting family with $t+1\leq C_t(\F)=c\leq k$. Then
		\begin{align*}
			|\F|\leq (t+1)(k-t+1)((k-t-1)!)^{r-1}\prod_{i=1}^{r}\binom{n_i-t-1}{k-t-1}.
		\end{align*}
	\end{Lemma}
	
	\begin{proof}
		Since $r\geq \log_2((t+1)(k-t+1)^2)$, we have $2\geq ((t+1)(k-t+1)^2)^{1/r}$. Thus
		\begin{align}\label{eq:20260107}
			\min_{1\leq i\leq r} n_i &\geq k+1=2+k-1\geq ((t+1)(k-t+1)^2)^{1/r}+k-1  \\
			& >((k-t+1)(k-t))^{1/r}+k-1.
		\end{align}
		It follows from Lemma \ref{lemma:F_c for matchings} that
		\begin{align*}
			|\F|\leq \binom{c}{t}(k-t+1)^{c-t}((k-c)!)^{r-1}\prod_{i=1}^{r}\binom{n_i-c}{k-c},
		\end{align*}
		where $t+1\leq c\leq k$. Define a function
		\begin{align*}
			L_4(c)=\binom{c}{t}(k-t+1)^{c-t}((k-c)!)^{r-1}\prod_{i=1}^{r}\binom{n_i-c}{k-c},
		\end{align*}
		where $t+1\leq c\leq k$. Then $|\F|\leq L_4(c)$. To complete the proof, it suffices to show that $L_4(c)\leq L_4(t+1)$ for any $t+1\leq c\leq k$.
		
		If $k=t+1$, then $c=t+1$, and so $L_4(c)= L_4(t+1)$. Assume that $t+1<c\leq k$. Then
		\begin{align*}
			\frac{L_4(c-1)}{L_4(c)}& =\frac{(c-t)\prod_{i=1}^{r}(n_i-c+1)}{(k-t+1)(k-c+1)c}
			=\frac{(1-\frac{t}{c})\prod_{i=1}^{r}(n_i-c+1)}{(k-t+1)(k-c+1)}\\
			&\geq \frac{(1-\frac{t}{t+2})\prod_{i=1}^{r}(n_i-c+1)}{(k-t+1)(k-c+1)}
			>\frac{\prod_{i=1}^{r}(n_i-c+1)}{(t+1)(k-t+1)(k-c+1)}\\
			&\geq \frac{(\min_{1\leq i\leq r} n_i-c+1)^r}{(t+1)(k-t+1)(k-c+1)}\geq \frac{(\min_{1\leq i\leq r} n_i-k+1)^r}{(t+1)(k-t+1)(k-c+1)}\\
			& > \frac{(\min_{1\leq i\leq r} n_i-k+1)^r}{(t+1)(k-t+1)(k-t)}>
			\frac{(\min_{1\leq i\leq r} n_i-k+1)^r}{(t+1)(k-t+1)^2}\geq 1,
		\end{align*}
		where the last inequality follows from \eqref{eq:20260107}. Therefore, $L_4(c)$ is decreasing on the interval $[t+1,k]$.
	\end{proof}
	
	\begin{proof}[\bf Proof of Theorem \ref{main_result_3}.]
		Let $\F_{0}$ be a $t$-star of $\M_k(n_1,\dots,n_r)$. By \eqref{eq:20251205} and Lemma \ref{lemma_proof_3}, it suffices to show that
		\begin{align*}
			|\F_0|=((k-t)!)^{r-1}\prod_{i=1}^{r}\binom{n_i-t}{k-t}> (t+1)(k-t+1)((k-t-1)!)^{r-1}\prod_{i=1}^{r}\binom{n_i-t-1}{k-t-1},
		\end{align*}
		i.e., to show that
		$$\prod_{i=1}^{r}(n_i-t) > (t+1)(k-t+1)(k-t).$$
		Since $r\geq \log_2((t+1)(k-t+1)^2)$, we have $2^r\geq (t+1)(k-t+1)^2$. Since $\min_{1\leq i\leq r} n_i-t\geq 2$, we have
		\begin{align*}
			\prod_{i=1}^{r}(n_i-t)\geq (\min_{1\leq i\leq r} n_i-t)^r \geq 2^r \geq (t+1)(k-t+1)^2 > (t+1)(k-t+1)(k-t).
		\end{align*}
		This completes the proof.
	\end{proof}

\section{Concluding remarks}\label{section_concluding}

This paper studies maximum-sized $t$-intersecting families of injections, partial permutations, generalized permutations, and $k$-matchings in $r$-partite $r$-uniform hypergraphs within a unified framework, and develops two complementary approaches to determine the maximum size of such families and to characterize the extremal examples. The first approach is based on a result in \cite{permutations_n_linear_t}. It can not only be used to obtain a compact consequence, namely Theorem \ref{thm:main-result-injections}, but also be generalized to Theorem \ref{thm:main-result}. The second approach uses the notion of $t$-covers and yields another class of results with a different emphasis. Theorem \ref{main_result_3} is of particular interest, as it requires only that the number of parts $r$ is sufficiently large, while each part can be as small as $k+1$.

A natural direction for future work is to extend our main results to the cross-$t$-intersecting setting. Families $\mathcal F_1\subseteq \binom{\Omega}{k_1},\ldots,\mathcal F_r\subseteq \binom{\Omega}{k_r}$ are cross-$t$-intersecting if every pair of members from distinct families intersects in at least $t$ elements; the size of such a structure may be measured by the sum or the product of the individual family sizes (see \cite{gp}). Keller, Lifshitz, Minzer, and Sheinfeld \cite{permutations_n_linear_t} have already given a cross-$t$-intersecting version of Theorem \ref{thm:KLMS}. Examining or improving the corresponding results for $k$-matchings in $r$-partite $r$-uniform hypergraphs remains an interesting direction for further investigation (cf. \cite{borg2014,borg2017max}).

%
%
%
%

\bibliographystyle{plain}

\end{document}